\documentclass[12pt]{article}
\usepackage{amssymb}
\usepackage{mathrsfs}
\usepackage[hypertex]{hyperref}
\usepackage{amsfonts}
\usepackage{graphicx}
\usepackage{mathptmx}
\usepackage{latexsym,amsmath,amssymb,amsfonts,amsthm}

\newtheorem{theorem}{Theorem}[section]

\newtheorem{corollary}[theorem]{Corollary}

\newtheorem{remark}[theorem]{Remark}

\begin{document}
\setcounter{page}{1}
\title{Estimates for sums of eigenvalues of the free plate with nonzero Poisson's ratio}
\author{Shan Li,~ Jing Mao$^{\ast}$}

\date{}
\protect\footnotetext{\!\!\!\!\!\!\!\!\!\!\!\!{$^{\ast}$Corresponding author}
}
\maketitle ~~~\\[-15mm]

\begin{center}
{\footnotesize Faculty of Mathematics and Statistics, \\
Key Laboratory of Applied Mathematics of Hubei Province, \\
Hubei University, Wuhan 430062, China \\
Email: jiner120@163.com }
\end{center}


\begin{abstract}
By using the Fourier transform, we successfully give Kr\"{o}ger-type
estimates for sums of eigenvalues of the free plate (under tension
and with nonzero Poisson's ratio) in terms of the dimension of the
ambient space, the volume of the domain, the tension parameter and
the Poisson's ratio.
 \end{abstract}

\markright{\sl\hfill S. Li, J. Mao \hfill}

\section{Introduction}
\renewcommand{\thesection}{\arabic{section}}
\renewcommand{\theequation}{\thesection.\arabic{equation}}
\setcounter{equation}{0}

For a bounded domain $\Omega$ in the Euclidean $n$-space
$\mathbb{R}^{n}$  with smooth boundary $\partial\Omega$, $n\geq2$,
the classical free membrane problem with the Neumann boundary
condition is actually the following boundary value problem (BVP for
short)
\begin{eqnarray} \label{1-1}
 \left\{
\begin{array}{ll}
\Delta u+\mu u=0 \qquad \qquad &\mathrm{in} ~ \Omega,\\
\frac{\partial u}{\partial\vec{v}}=0 \qquad \qquad &\mathrm{on} ~
\partial \Omega,
\end{array}
\right.
\end{eqnarray}
where $\Delta$ is the Lapacian and $\vec{v}$ denotes the outward
unit normal vector of $\partial\Omega$. It is well-known that
$\Delta$ in (\ref{1-1}) only has discrete spectrum and all the
elements (i.e., eigenvalues), with finite multiplicity, in its
spectrum can be listed non-decreasingly as follows
\begin{eqnarray*}
0=\mu_{1}(\Omega)<\mu_{2}(\Omega)\leq\mu_{3}(\Omega)\leq\cdots\uparrow\infty.
\end{eqnarray*}
For the BVP (\ref{1-1}), there are so many interesting and existing
estimates for Neumann eigenvalues $\mu_{i}(\Omega)$. Here, we would
like to mention the following two facts:

\begin{itemize}

\item (Szeg\H{o} \cite{gs1,gs2}, Weinberg \cite{hfw})  \emph{Among all domains with fixed volume,
the lowest nonzero Neumann eigenvalue $\mu_{2}(\Omega)$ is maximized
by a ball}.

\item (Kr\"{o}ger \cite{pk}) \emph{Estimates
\begin{eqnarray*}
\sum\limits_{i=1}^{m}\mu_{i}(\Omega)\leq
(2\pi)^{2}\frac{n}{n+2}\left(w_{n}|\Omega|\right)^{\frac{2}{n}}m^{\frac{n+2}{n}},~~~~~m\geq1
\end{eqnarray*}
and
\begin{eqnarray*}
\mu_{m+1}(\Omega)\leq
(2\pi)^{2}\left(\frac{n+2}{2w_{n}|\Omega|}\right)^{\frac{2}{n}}m^{\frac{2}{n}},~~~~~m\geq0
\end{eqnarray*}
hold, where $|\Omega|$, $w_{n}$ denote the volume of $\Omega$ and
the volume of the unit ball in $\mathbb{R}^{n}$, respectively}.

\end{itemize}

Consider the following eigenvalue problem of free plate under
tension
\begin{eqnarray} \label{1-2}
 \left\{
\begin{array}{lll}
\Delta^{2}u-\tau\Delta u= \Lambda u \qquad \qquad &\mathrm{in} ~ \Omega,\\
\frac{\partial^{2}u}{\partial\vec{v}^{2}}=0\qquad \qquad
&\mathrm{on} ~
\partial \Omega,\\
\tau\frac{\partial
u}{\partial\vec{v}}-\mathrm{div}_{\partial\Omega}\left({\mathrm{Proj}}_{\partial\Omega}\left[(D^{2}u)\vec{v}\right]\right)-\frac{\partial\Delta
u}{\partial\vec{v}}=0 \qquad \qquad &\mathrm{on} ~
\partial \Omega,
\end{array}
\right.
\end{eqnarray}
where $\Delta^{2}$ is the bi-Laplace operator in
$\Omega\subset\mathbb{R}^n$, $\mathrm{div}_{\partial\Omega}$ is the
surface divergence on $\partial\Omega$, the operator
${\mathrm{Proj}}_{\partial\Omega}$ projects onto the space tangent
to $\partial\Omega$, $D^{2}u$ denotes the Hessian matrix, and other
same symbols have the same meanings as those in (\ref{1-1}).
Physically, when $n=2$, $\Omega$ is the shape of a homogeneous,
isotropic plate, and the parameter $\tau$ is the ratio of lateral
tension to flexural rigidity of the plate. Positive $\tau$
corresponds to a plate under tension, while negative $\tau$ gives us
a plate under compression. Chasman \cite[Section 4]{lmc1} proved
that if $\tau\geq0$, the operator $\Delta^{2}-\tau\Delta$ in the BVP
(\ref{1-2}) has the discrete spectrum and all the eigenvalues, with
finite multiplicity, in this spectrum can be listed non-decreasingly
as follows
\begin{eqnarray*}
0=\Lambda_{1}(\Omega)<\Lambda_{2}(\Omega)\leq\Lambda_{3}(\Omega)\leq\cdots\uparrow\infty.
\end{eqnarray*}
Moreover, Chasman \cite[Theorem 1]{lmc1} showed that:

\begin{itemize}
\item \emph{Among all domains with fixed volume, the lowest nonzero
 eigenvalue $\Lambda_{2}(\Omega)$ for a free plate under
tension (i.e., $\tau>0$) is maximized by a ball}.
\end{itemize}

Could one expect Kr\"{o}ger-type estimates for $\Lambda_{i}(\Omega)$
of the BVP (\ref{1-2}) provided $\tau\geq0$ ?

Very recently, Brandolini, Chiacchio and Langford \cite[Theorem 1
and Corollary 2]{bcl} gave a positive answer to the above question.

After getting the important isoperimetric inequality for
$\Lambda_{2}(\Omega)$ of a free plate under tension, Chasman
considered the following eigenvalue problem of free plate under
tension and \emph{with nonzero Poisson's ratio}
\begin{eqnarray} \label{1-3}
 \left\{
\begin{array}{lll}
\Delta^{2}u-\tau\Delta u= \Gamma u \qquad \qquad &\mathrm{in} ~ \Omega,\\
(1-\sigma)\frac{\partial^{2}u}{\partial\vec{v}^{2}}+\sigma\Delta
u=0\qquad \qquad &\mathrm{on} ~
\partial \Omega,\\
\tau\frac{\partial
u}{\partial\vec{v}}-(1-\sigma)\mathrm{div}_{\partial\Omega}\left({\mathrm{Proj}}_{\partial\Omega}\left[(D^{2}u)\vec{v}\right]\right)-\frac{\partial\Delta
u}{\partial\vec{v}}=0 \qquad \qquad &\mathrm{on} ~
\partial \Omega,
\end{array}
\right.
\end{eqnarray}
where $\sigma$ is the Poisson's ratio\footnote{ Poisson's ratio is a
property of the material of the plate. Usually, if a material is
stretched in one direction, it contracts in the orthogonal
directions. In such situation, the value $\sigma$ is a ratio of the
strains. However, some materials expand in the orthogonal directions
rather than contracting, and then have $\sigma<0$, which leads to
the situation that they are called auxetic.} and other same symbols
have the same meanings as those in (\ref{1-2}). Typically, $\sigma$
is taken to be $\sigma\in[0,0.5]$ for real-world materials, although
 a class of materials known as auxetics have negative Poisson's
ratio. However, in order to be assured of coercivity of the
sesquilinear form $a(u,v)$ defined  by  (\ref{2-1}) (see also
\cite[Section 4]{lmc2}), one needs to require
$\sigma\in(-1/(n-1),1)$. Chasman \cite[Section 4]{lmc2} explained
that if $\tau\geq0$ and $\sigma\in(-1/(n-1),1)$, the operator
$\Delta^{2}-\tau\Delta$ in the BVP (\ref{1-3}) has the discrete
spectrum and all the eigenvalues\footnote{ See also \textbf{FACT} in
Section \ref{s2} for the reason why $\Gamma_{i}(\Omega)$ is
nonnegative provided $\tau\geq0$ and $\sigma\in(-1/(n-1),1)$.}, with
finite multiplicity, in this spectrum can be listed non-decreasingly
as follows
\begin{eqnarray*}
0=\Gamma_{1}(\Omega)<\Gamma_{2}(\Omega)\leq\Gamma_{3}(\Omega)\leq\cdots\uparrow\infty.
\end{eqnarray*}
 He also proved that similar to
$\Lambda_{2}(\Omega)$, the ball with the same volume maximizes
$\Gamma_{2}(\Omega)$ if the free plate is under tension and  one of
the followings holds:
\begin{itemize}

\item $n=2$ and $\sigma>-51/97$ or $\tau\geq3(\sigma-1)/(\sigma+1)$,

\item $n=3$,

\item $n\geq4$ and $\sigma\leq0$ or $\tau\geq(n+2)/2$.

\end{itemize}
However, numerical and analytic evidences suggest that this fact
should hold for $\tau>0$, $\sigma\in(-1/(n-1),1)$ -- see
\cite[Section 8]{lmc2} for details. Based on this, Chasman
\cite{lmc2} conjectured:

\begin{itemize}
\item \emph{Among all domains with fixed volume, the lowest nonzero
 eigenvalue $\Gamma_{2}(\Omega)$ for a free plate under
tension, with Poisson's ratio $\sigma\in(-1/(n-1),1)$, is maximized
by a ball}.
\end{itemize}
This conjecture is open, and the best partial answers so far are due
to Chasman \cite{lmc1,lmc2}.

Inspired by Brandolini-Chiacchio-Langford's Kr\"{o}ger-type
estimates for $\Lambda_{i}(\Omega)$ and Chasman's Szeg\H{o}-Weinberg
type isoperimetric inequalities for  $\Lambda_{2}(\Omega)$ and
$\Gamma_{2}(\Omega)$, one might ask

\vspace{3mm}

\textbf{Question}. \emph{Is it possible to get Kr\"{o}ger-type
estimates for $\Gamma_{i}(\Omega)$ of the BVP (\ref{1-3}) ?}

\vspace{3mm}

The answer is positive. In fact, we can prove:

\begin{theorem} \label{maintheorem}
Let $\Omega \subset \mathbb{R}^n$ be a smooth bounded domain, and
let ${\Gamma_{j}}(\Omega)$ be the $j$-th eigenvalue of the BVP
(\ref{1-3}). If $\tau\geq0$ and $\sigma\in(-1/(n-1),1)$, then
\begin{eqnarray*}
\sum_{j=1}^{m}\Gamma_{j}(\Omega)\leq(2\pi)^{4}\frac{n}{n+4}\left(\frac{1}{w_{n}|\Omega|}\right)^{\frac{4}{n}}m^{\frac{n+4}{4}}+
\tau(2\pi)^{2}\frac{n}{n+2}
\left(\frac{1}{w_{n}|\Omega|}\right)^{\frac{2}{n}}m^{\frac{n+2}{n}},
 \end{eqnarray*}
 where, as before, $|\Omega|$, $w_{n}$ denote the volume of $\Omega$ and
the volume of the unit ball in $\mathbb{R}^{n}$, respectively.
\end{theorem}

Inspired by the proof of Theorem \ref{maintheorem} shown in Section
\ref{s3} below, one can also get the following estimates.

\begin{corollary} \label{corollary}
Under the assumptions in Theorem \ref{maintheorem}, we have:

(1) If $\tau>0$, then
 \begin{eqnarray*}
\Gamma_{m+1}(\Omega)\leq
\min\limits_{r>2\pi\left(\frac{m}{w_{n}|\Omega|}\right)^{\frac{1}{n}}}\frac{nw_{n}|\Omega|(\frac{r^{n+4}}{n+4}+\tau\frac{r^{n+2}}{n+2})}{{w_{n}}|\Omega|r^{n}-(2\pi)^{n}m},
\qquad m\geq0.
 \end{eqnarray*}

(2) If $\tau=0$, then
\begin{eqnarray*}
\Gamma_{m+1}(\Omega)\leq
(2\pi)^{4}\left(\frac{m(n+4)}{4w_{n}|\Omega|}\right)^{\frac{4}{n}},
\qquad m\geq0.
\end{eqnarray*}
\end{corollary}

\begin{remark}
\rm{ (1) When $\tau\geq0$, $\sigma\in[0,1)$, the above three estimates have been shown in \cite[Theorem 1.4]{dmwxy}. Speaking in other words,
Theorem \ref{maintheorem} and Corollary \ref{corollary} here can be seen as an extension of \cite[Theorem 1.4]{dmwxy}.\\
 (2) Chasman's conjecture  mentioned
above (i.e., the Szeg\H{o}-Weinberg type isoperimetric inequality
for $\Gamma_{2}(\Omega)$) and his partial answer to this conjecture
tell us that conclusions for the BVP (\ref{1-2}) might not be
transferred to the case of the BVP (\ref{1-3}) directly, that is to
say, \emph{for the eigenvalue problem of free plate under tension,
sometimes, there exists difference between the zero Poisson's ratio
case and the nonzero case}. Based on this fact, it is attractive
that we can also get  Kr\"{o}ger-type estimates for the BVP
(\ref{1-3})
under suitable assumptions. \\
(3) It is surprising that the Kr\"{o}ger-type upper bounds given in
Theorem \ref{maintheorem} and Corollary \ref{corollary} are the same
as those in \cite[Theorem 1 and Corollary 2]{bcl} for eigenvalues of
the BVP (\ref{1-2}). \emph{This gives an example that sometimes
there does not exist obvious difference between the zero Poisson's
ratio case and the nonzero case}. }
\end{remark}

\section{Boundary conditions} \label{s2}
\renewcommand{\thesection}{\arabic{section}}
\renewcommand{\theequation}{\thesection.\arabic{equation}}
\setcounter{equation}{0}

In this section, we would like to give an explanation to the
rationality of boundary conditions such that one can understand the
BVP (\ref{1-3}) well.

As shown in \cite{lmc2}, the sesquilinear form associated with the
free plate problem (\ref{1-3}) is defined as follows
\begin{eqnarray} \label{2-1}
a(u,v)=\int_{\Omega}\left[(1-\sigma)\sum_{i,j=1}^{n}\overline{u_{x_{i}x_{j}}}v_{x_{i}x_{j}}+\sigma\overline{\Delta
u}\Delta v+\tau \overline{Du}\cdot Dv\right] dx ,\qquad u,v\in
H^{2}(\Omega),
\end{eqnarray}
where $D$ is the gradient operator on $\Omega$, and other symbols
have the same meanings as before. Moreover, for the BVP (\ref{1-3}),
its generalized Rayleigh quotient $Q[u]$ has the form
\begin{eqnarray*}
Q[u]:&=&\frac{\int_{\Omega}\left[(1-\sigma)|D^{2}u|^{2}+\sigma(\Delta
u)^{2}+\tau|Du|^{2}\right] dx}{\int_{\Omega}|Du|^{2}dx}\\
&=&\frac{a(u,u)}{\|u\|^{2}_{L^2}}.
\end{eqnarray*}
For the rest part of this section, we would consider the BVP
(\ref{1-3}) in the case $\tau\geq0$ and $\sigma\in(-1/(n-1),1)$. In
fact, if $\tau\geq0$ and $\sigma\in(-1/(n-1),1)$, Chasman
\cite[Section 4]{lmc2} showed that $a(\cdot,\cdot)$ is coercive, all
eigenvalues $\Gamma_{i}(\Omega)$ is nonnegative and the
corresponding eigenfunctions are \emph{real-valued} and smooth on
$\overline{\Omega}$. Hence, under the assumptions  $\tau\geq0$ and
$\sigma\in(-1/(n-1),1)$, one can neglect the effect of conjugate
part of the form $a(u,v)$, that is to say, if $\tau\geq0$ and
$\sigma\in(-1/(n-1),1)$, one can rewrite $a(u,v)$ as
\begin{eqnarray*}
a(u,v)=\int_{\Omega}\left[(1-\sigma)\sum_{i,j=1}^{n}u_{x_{i}x_{j}}v_{x_{i}x_{j}}+\sigma\Delta
u\Delta v+\tau Du\cdot Dv\right] dx ,\qquad u,v\in H^{2}(\Omega)
\end{eqnarray*}
directly. In fact, about the BVP (\ref{1-3}), one has the following
fundamental fact:

\begin{itemize}

\item \textbf{FACT}.
\emph{Let $u_{i}\in H^{2}(\Omega)$ be the eigenfunction of the
$i$-th eigenvalue $\Gamma_{i}(\Omega)$, $i=0,1,2,\cdots,m$,
$m=0,1,2,\cdots$. Then one has
\begin{eqnarray*}
0\leq\Gamma_{m+1}(\Omega)=\inf\left\{\frac{a(u,u)}{\|u\|^{2}_{L^2}}\Bigg{|}u\in
H^{2}(\Omega),\int_{\Omega}uu_{i}dx=0\right\}
 \end{eqnarray*}
 provided $\tau\geq0$ and
$\sigma\in(-1/(n-1),1)$.}
 \begin{proof}
The characterization of $\Gamma_{m+1}(\Omega)$ (i.e., the equality
case) can be easily obtained by variational method and the fact that
eigenfunctions belonging to different eigenvalues are orthogonal
with each other. If $\tau\geq0$, $\sigma\in[0,1)$, then the form
$a(u,u)$ is obviously nonnegative, which implies the nonnegativity
of $\Gamma_{m+1}(\Omega)$ naturally. If  $\tau\geq0$,
$\sigma\in(-1/(n-1),0)$, then by using \cite[FACT1]{lmc2}, one has
\begin{eqnarray*}
a(u,u)&\geq& (1-\sigma)\int_{\Omega}|D^{2}u|^{2}dx+
n\sigma\int_{\Omega}|D^{2}u|^{2}dx+\tau\int_{\Omega}|Du|^{2}dx\\
&=&\left(1+(n-1)\sigma\right)\int_{\Omega}|D^{2}u|^{2}dx+\tau\int_{\Omega}|Du|^{2}dx\geq0,
\end{eqnarray*}
which implies the nonnegativity of $\Gamma_{m+1}(\Omega)$. Our
\textbf{FACT} follows.
 \end{proof}

\end{itemize}
If $u$ is the weak solution of (\ref{1-3}), then we have
\begin{equation}  \label{2-2}
a(u,v)=\Gamma\int_{\Omega}uv dx,\qquad u,v\in H^{2}(\Omega).
\end{equation}
By direct calculation, one can obtain
 \begin{eqnarray*}
 \int_{\Omega}\overline{Du}\cdot Dv
dx=\int_{\partial\Omega}v\frac{\partial
u}{\partial\vec{v}}ds-\int_{\Omega}v\Delta udx
\end{eqnarray*}
and
 \begin{eqnarray*}
&&\qquad\int_{\Omega}\sum_{i,j=1}^{n}\overline{u_{x_{i}x_{j}}}v_{x_{i}x_{j}}dx\\
&&={\int_{\partial\Omega}\left[Dv\cdot((D^{2}u)\cdot\vec{v})-v\frac{\partial(\Delta u)}{\partial\vec{v}}\right]ds+\int_{\Omega}(\Delta^{2}u)v dx}\\
&&=\int_{\partial\Omega}\left[\frac{\partial
v}{\partial\vec{v}}\cdot\frac{\partial^{2}u}{\partial\vec{v}^{2}}-v
\mathrm{div}_{\partial\Omega}\left({\mathrm{Proj}}_{\partial\Omega}\left[(D^{2}u)\vec{v}\right]\right)-v\frac{\partial(\Delta
u)}{\partial\vec{v}}\right]ds+\int_{\Omega}(\Delta^{2}u)v dx.
 \end{eqnarray*}
Together with (\ref{2-1}) and (\ref{2-2}), we have
\begin{eqnarray} \label{2-3}
&&\int_{\Omega}[(1-\sigma)(\Delta^{2}u)v-\tau v\Delta u-\Gamma uv]dx+\sigma\int_{\Omega}\Delta v\overline{\Delta u}dx+\nonumber\\
&&\qquad\qquad
\int_{\partial\Omega}\left\{(1-\sigma)\left[\frac{\partial
v}{\partial\vec{v}}\cdot\frac{\partial^{2}u}{\partial\vec{v}^{2}}-v
\mathrm{div}_{\partial\Omega}\left({\mathrm{Proj}}_{\partial\Omega}\left[(D^{2}u)\vec{v}\right]\right)-
v\frac{\partial(\Delta u)}{\partial\vec{v}}\right]+\tau v\frac{\partial u}{\partial\vec{v}}\right\}ds\nonumber\\
&&=0.
\end{eqnarray}
Besides, applying the divergence theorem, one can easily get
\begin{eqnarray} \label{2-4}
\int_{\Omega}\Delta v\overline{\Delta u}dx&=&\int_{\partial\Omega}\Delta u\frac{\partial v}{\partial\vec{v}}ds-\int_{\Omega}D(\Delta u)\cdot Dv dx\nonumber\\
&=&\int_{\partial\Omega}\Delta u\frac{\partial
v}{\partial\vec{v}}ds-\left[\int_{\partial\Omega}v\frac{\partial(\Delta
u)}{\partial\vec{v}}ds-\int_{\Omega}
(\Delta^{2}u)v dx\right]\nonumber\\
&=&\int_{\partial\Omega}\left(\Delta u\frac{\partial
v}{\partial\vec{v}}-v\frac{\partial(\Delta
u)}{\partial\vec{v}}\right)ds+\int_{\Omega} (\Delta^{2}u)v dx.
\end{eqnarray}
Combining (\ref{2-3}) and (\ref{2-4}) yields
\begin{eqnarray} \label{2-5}
&&\int_{\Omega}\left(\Delta^{2}u-\tau \Delta u-\Gamma
u\right)vdx+\int_{\partial\Omega}\frac{\partial v}{\partial\vec{v}}
\left[(1-\sigma)\frac{\partial^{2}u}{\partial\vec{v}^{2}}+\sigma\Delta u\right]ds+\nonumber\\
&&\qquad \int_{\partial\Omega}v\left[\tau\frac{\partial
u}{\partial\vec{v}}-(1-\sigma)
\mathrm{div}_{\partial\Omega}\left({\mathrm{Proj}}_{\partial\Omega}\left[(D^{2}u)\vec{v}\right]\right)-\frac{\partial
(\Delta u)}{\partial\vec{v}}\right]ds\nonumber\\
&&=0.
\end{eqnarray}
In (\ref{2-5}), taking $v\in C^{\infty}_{0}(\Omega)$ to be a test
function and observing that any smooth function $v\in
C^{\infty}(\partial\Omega)$ can be extended to
$C^{\infty}(\overline{\Omega})$ with $\frac{\partial
v}{\partial\vec{v}}=0$ along the boundary $\partial\Omega$, one
knows that the equation $\Delta^{2}u-\tau \Delta u-\Gamma u=0$ holds
in $\Omega$, together with two boundary conditions
\begin{eqnarray*}
(1-\sigma)\frac{\partial^{2}u}{\partial\vec{v}^{2}}+\sigma\Delta u=0
\end{eqnarray*}
and
\begin{eqnarray*}
\tau\frac{\partial u}{\partial\vec{v}}-(1-\sigma)
\mathrm{div}_{\partial\Omega}\left({\mathrm{Proj}}_{\partial\Omega}\left[(D^{2}u)\vec{v}\right]\right)-\frac{\partial
(\Delta u)}{\partial\vec{v}}=0
\end{eqnarray*}
in $\partial\Omega$, which is the BVP (\ref{1-3}) exactly.

\section{Proof of the main result} \label{s3}
\renewcommand{\thesection}{\arabic{section}}
\renewcommand{\theequation}{\thesection.\arabic{equation}}
\setcounter{equation}{0}

Now, we would like to use the method introduced in \cite{pk} to
derive the estimate given in Theorem \ref{maintheorem}.

\begin{proof}[Proof of Theorem \ref{maintheorem}]
Let $\phi_{1}, \phi_{2}, \cdots,\phi_{m}$ represent the orthogonal
eigenfunctions in $H^{2}(\Omega)$ corresponding to
$\Gamma_{1}(\Omega),\Gamma_{2}(\Omega),\cdots,\Gamma_{m}(\Omega)$.
Define
\begin{eqnarray*}
\Phi(x,y)=\sum_{j=1}^{m}\phi_{j}(x)\phi_{j}(y), \qquad x,y\in\Omega.
\end{eqnarray*}
Let $\widehat{\Phi}(z,y)$ be the Fourier transform of $\Phi$ in the
variable $x$, that is to say,
\begin{eqnarray*}
\widehat{\Phi}(z,y)=\frac{1}{(2\pi)^{\frac{n}{2}}}\int_{\Omega}\Phi(x,y)e^{ix\cdot
z}dx.
\end{eqnarray*}
Hence, one has
\begin{eqnarray*}
(2\pi)^{\frac{n}{2}}\widehat{\Phi}(z,y)=\sum_{j=1}^{m}\phi_{j}(y)\int_{\Omega}e^{iz\cdot
x}\cdot\phi_{j}(x)dx.
\end{eqnarray*}
Set $h_{z}(y)=e^{iy\cdot z}$ and define
$\rho(z,y):=h_{z}(y)-(2\pi)^{\frac{n}{2}}\widehat{\Phi}(z,y)$. It is
easy to verify $\rho(z,y)\in H^{2}(\Omega)$. Using $\rho(z,y)$ as
the test function in the generalized Rayleigh quotient $Q$ yields
\begin{eqnarray*}
&&\Gamma_{m+1}(\Omega)\leq
Q[\rho(z,y)]=\frac{\int_{\Omega}\left[(1-\sigma)|D^{2}\rho|^{2}+\sigma(\Delta
\rho)^{2}+\tau|D\rho|^{2}\right]dy}{\int_{\Omega}\rho^{2}dy}\\
&&\qquad=\frac{(1-\sigma)\int_{\Omega}\sum\limits_{j,k=1}^{n}|\rho(z,y)_{y_{j}y_{k}}|^{2}dy+\sigma
\int_{\Omega}\sum\limits_{j=1}^{n}|\rho(z,y)_{y_{j}y_{j}}|^{2}dy+\tau\int_{\Omega}\sum\limits_{j=1}^{n}|\rho(z,y)_{y_{j}}|^{2}dy
}{\int_{\Omega}\rho^{2}dy}.
\end{eqnarray*}
Multiplying both sides of the above inequality by the denominator
and integrating over $B_{r}=\{z\in \mathbb{R}^{n}||z|<r\}$ result
into
 \begin{eqnarray*}
&&\Gamma_{m+1}(\Omega)\leq
\inf\limits_{r}\Bigg{\{}\frac{(1-\sigma)\int_{B_{r}}\int_{\Omega}\sum\limits_{j,k=1}^{n}|\rho(z,y)_{y_{j}y_{k}}|^{2}dydz+\sigma
\int_{B_{r}}\int_{\Omega}\sum\limits_{j=1}^{n}|\rho(z,y)_{y_{j}y_{j}}|^{2}dydz}{\int_{B_{r}}\int_{\Omega}\rho^{2}dydz}+\\
&&\qquad \qquad \qquad \qquad \qquad \qquad\qquad \qquad \qquad
\qquad \qquad
\frac{\tau\int_{B_{r}}\int_{\Omega}\sum\limits_{j=1}^{n}|\rho(z,y)_{y_{j}}|^{2}dydz
}{\int_{B_{r}}\int_{\Omega}\rho^{2}dydz}\Bigg{\}}\\
&& \qquad \qquad  := \inf\limits_{r}\left\{\frac{N}{D}\right\},
\end{eqnarray*}
where the infimum is taken over the set
$\left\{r|r>2\pi\left(\frac{m}{w_{n}|\Omega|}\right)^{\frac{1}{n}}\right\}$.
In order to get the conclusion, we need to estimate the numerator
$N$ and the  denominator $D$. By (11) of \cite{bcl}, one has
\begin{eqnarray} \label{3-1}
D=w_{n}|\Omega|r^{n}-(2\pi)^{n}\sum_{j=1}^{m}\int_{B_{r}}|\widehat{\phi_{j}}(z)|^{2}dz.
\end{eqnarray}
Rewrite $N$ as follows
 \begin{eqnarray*}
N=I_{1}+I_{2}+I_{3},
 \end{eqnarray*}
where
\begin{eqnarray*}
&&I_{1}=\sum\limits_{j,k=1}^{n}{\int_{B_{r}}}\int_{\Omega}(1-\sigma)|h_{z}(y)_{y_{j}y_{k}}|^{2}dydz+\sum\limits_{j=1}^{n}{\int_{B_{r}}}
\int_{\Omega}\sigma|h_{z}(y)_{y_{j}y_{j}}|^{2}dydz+\\
&&\qquad\qquad
\tau\sum\limits_{j=1}^{n}{\int_{B_{r}}}\int_{\Omega}|h_{z}(y)_{y_{j}}|^{2}dydz,
\end{eqnarray*}
\begin{equation*}
\begin{split}
I_{2}=&-2(1-\sigma)(2\pi)^{\frac{n}{2}}\mathrm{Re}\left\{\sum_{j,k=1}^{n}{\int_{B_{r}}}\int_{\Omega}h_{z}(y)_{y_{j}y_{k}}\overline{\widehat{\Phi}(z,y)_{y_{j}y_{k}}}dydz\right\}-
\\
&2\sigma(2\pi)^{\frac{n}{2}}\mathrm{Re}\left\{\sum_{j=1}^{n}{\int_{B_{r}}}\int_{\Omega}h_{z}(y)_{y_{j}y_{j}}\overline{\widehat{\Phi}(z,y)_{y_{j}y_{j}}}dydz\right\}-
\\ & 2\tau(2\pi)^{\frac{n}{2}}\mathrm{Re}\left\{\sum_{j=1}^{n}{\int_{B_{r}}}\int_{\Omega}h_{z}(y)_{y_{j}}\overline{\widehat{\Phi}(z,y)_{y_{j}}}dydz\right\}, \\
I_{3}=&(2\pi)^{n}\sum_{j,k=1}^{n}{\int_{B_{r}}}\int_{\Omega}(1-\sigma)|\widehat{\Phi}(z,y)_{y_{j}y_{k}}|^{2}dydz+ \\
&(2\pi)^{n}\sum_{j=1}^{n}{\int_{B_{r}}}\int_{\Omega}\sigma|\widehat{\Phi}(z,y)_{y_{j}y_{j}}|^{2}dydz+
(2\pi)^{n}\tau\sum_{j=1}^{n}{\int_{B_{r}}}\int_{\Omega}|\widehat{\Phi}(z,y)_{y_{j}}|^{2}dydz.
\end{split}
\end{equation*}
For $I_{1}$, using the facts
 \begin{eqnarray*}
|h_{z}(y)|=|e^{iy\cdot
z}|=1,~|h_{z}(y)_{y_{j}}|=|z_{j}|,~|h_{z}(y)_{y_{j}y_{k}}|=|z_{j}||z_{k}|,
 \end{eqnarray*}
 we have
 \begin{eqnarray} \label{3-2}
I_{1}&=&\int_{B_{r}}\int_{\Omega}\left[(1-\sigma)|z|^{4}+\sigma|z|^{4}+\tau|z|^{2}\right]dydz\nonumber\\
&=&nw_{n}|\Omega|
\left(\frac{r^{n+4}}{n+4}+\tau\frac{r^{n+2}}{n+2}\right).
 \end{eqnarray}
For $I_{2}$, we have
\begin{equation*}
\begin{split}
I_{2}&=-2(1-\sigma)(2\pi)^{\frac{n}{2}}\mathrm{Re}\left\{\sum_{j,k=1}^{n}\int_{B_r}\int_{\Omega}h_{z}(y)_{y_{j}y_{k}} \overline{\widehat{\Phi}(z,y)_{y_{j}y_{j}}}dydz \right\}\\
&\quad - 2\sigma(2\pi)^{\frac{n}{2}}\mathrm{Re}\left\{\sum_{j=1}^{n}{\int_{B_{r}}}\int_{\Omega}h_{z}(y)_{y_{j}y_{j}}\overline{\widehat{\Phi}(z,y)_{y_{j}y_{j}}}dydz \right\}\\
&\quad  -2\tau(2\pi)^{\frac{n}{2}}\mathrm{Re}\left\{\sum_{j=1}^{n}{\int_{B_{r}}}\int_{\Omega}h_{z}(y)_{y_{j}}\overline{\widehat{\Phi}(z,y)_{y_{j}}}dydz\right\}\\
&=-2(2\pi)^{\frac{n}{2}}\mathrm{Re}\left\{\int_{B_{r}}\int_{\Omega}\left(\sum\limits_{j,k=1}^{n}h_{z}(y)_{y_{j}y_{k}}\overline{\widehat{\Phi}(z,y)_{y_{j}y_{k}}}+ \tau\sum\limits_{j=1}^{n}h_{z}(y)_{y_{j}}\overline{\widehat{\Phi}(z,y)_{y_{j}}}\right)dydz\right\} \\
&\quad
+2\sigma(2\pi)^{\frac{n}{2}}\mathrm{Re}\left\{\int_{B_r}\int_{\Omega}\left(\sum_{j,k=1}^{n}h_{z}(y)_{y_{j}y_{k}}\overline{\widehat{\Phi}(z,y)_{y_{j}y_{k}}}-
\sum_{j=1}^{n}h_{z}(y)_{y_{j}y_{j}}\overline{\widehat{\Phi}(z,y)_{y_{j}y_{j}}}\right)dydz \right\}\\
&=2\sigma(2\pi)^{\frac{n}{2}}\mathrm{Re}\left\{\int_{B_r}\int_{\Omega}\left(\sum_{j,k=1}^{n}h_{z}(y)_{y_{j}y_{k}}\overline{\widehat{\Phi}(z,y)_{y_{j}y_{k}}}
    - \sum_{j=1}^{n}h_{z}(y)_{y_{j}y_{j}}\overline{\widehat{\Phi}(z,y)_{y_{j}y_{j}}}\right)dydz\right \}\\
&\quad-2(2\pi)^{n}\sum\limits_{j=1}^{m}\int_{B_{r}}\Gamma_{j}(\Omega)|\widehat{\phi_{j}}(z)|^{2}dz.
\end{split}
\end{equation*}
Since
$\widehat{\Phi}(z,y)=\sum_{j=1}^{m}\widehat{\phi}_{j}(z)\phi_{j}(y)$,
one has
\begin{equation*}
\begin{split}
&2\sigma(2\pi)^{\frac{n}{2}}\mathrm{Re}\left\{\int_{B_r}\int_{\Omega}\left(\sum_{j,k=1}^{n}h_{z}(y)_{y_{j}y_{k}}\overline{\widehat{\Phi}(z,y)_{y_{j}y_{k}}}- \sum_{j=1}^{n}h_{z}(y)_{y_{j}y_{j}}\overline{\widehat{\Phi}(z,y)_{y_{j}y_{j}}}\right)dydz \right\}\\
&\quad=2\sigma(2\pi)^{\frac{n}{2}}\mathrm{Re}\Bigg{\{}\int_{B_r}\int_{\Omega}\Bigg{(}h_{z}(y)\overline{\Delta_{y}^{2}\widehat{\Phi}(z,y)}-h_{z}(y)
\overline{\Delta_{y}^{2}\widehat{\Phi}(z,y)}\Bigg{)}dydz \Bigg{\}}\\
&\quad=0.
\end{split}
\end{equation*}
Therefore, we can obtain
\begin{eqnarray} \label{3-3}
 I_{2}=-2(2\pi)^n\sum_{j=1}^m \int_{B_r} \Gamma_{j}(\Omega) |\widehat{\phi_j}(z)|^2dz.
\end{eqnarray}
Finally, for $I_{3}$, one has
\begin{equation*}
\begin{split}
I_{3}=&(2\pi)^{n}{\int_{B_{r}}}\int_{\Omega}\left(\sum_{j,k=1}^{n}(1-\sigma)|\widehat{\Phi}(z,y)_{y_{j}y_{k}}|^{2}+\sigma\sum_{j=1}^{n}|\widehat{\Phi}(z,y)_{y_{j}y_{j}}|^{2}+
\tau\sum_{j=1}^{n}|\widehat{\Phi}(z,y)_{y_{j}}|^{2}\right)dydz\\
=&(2\pi)^{n}\sum_{l=1}^{m}\Gamma_{l}(\Omega)\int_{B_{r}}|\widehat{\phi}_{l}(z)|^{2}dz-
(2\pi)^{n}\sigma{\int_{B_{r}}}\int_{\Omega}\sum_{j,k=1}^{n}|\widehat{\Phi}(z,y)_{y_{j}y_{k}}|^{2}dydz+\\
&\qquad (2\pi)^{n}\sigma{\int_{B_{r}}}\int_{\Omega}\sum_{j=1}^{n}
|\widehat{\Phi}(z,y)_{y_{j}y_{j}}|^{2}dydz.
\end{split}
\end{equation*}
On the other hand,
\begin{equation*}
\begin{split}
&-(2\pi)^{n}\sigma{\int_{B_{r}}}\int_{\Omega}\sum_{j,k=1}^{n}|\widehat{\Phi}(z,y)_{y_{j}y_{k}}|^{2}dydz+
(2\pi)^{n}\sigma{\int_{B_{r}}}\int_{\Omega}\sum_{j=1}^{n}|\widehat{\Phi}(z,y)_{y_{j}y_{j}}|^{2}dydz\\
=&-(2\pi)^{n}\sigma{\int_{B_{r}}}\int_{\Omega}\sum_{j,k=1}^{n}\widehat{\Phi}(z,y)_{y_{j}y_{k}}\overline{\widehat{\Phi}(z,y)_{y_{j}y_{k}}}dydz+
(2\pi)^{n}\sigma{\int_{B_{r}}}\int_{\Omega}\sum_{j=1}^{n}\widehat{\Phi}(z,y)_{y_{j}y_{j}}\overline{\widehat{\Phi}(z,y)_{y_{j}y_{j}}}dydz\\
=&-(2\pi)^{n}\sigma{\int_{B_{r}}}\int_{\Omega}\left(\widehat{\Phi}(z,y)\overline{\Delta_{y}^{2}\widehat{\Phi}(z,y)_{y_{j}y_{j}}}-
\widehat{\Phi}(z,y)\overline{\Delta_{y}^{2}\widehat{\Phi}(z,y)_{y_{j}y_{j}}}\right)dydz\\
=& 0
\end{split}
\end{equation*}
Hence, we can deduce that
\begin{eqnarray} \label{3-4}
I_{3}=(2\pi)^{n}\sum_{l=1}^{m}\Gamma_{l}(\Omega)\int_{B_{r}}|\widehat{\phi}_{l}(z)|^{2}dz.
\end{eqnarray}
Combining (\ref{3-2})-(\ref{3-4}), it is easy to know
\begin{eqnarray*}
N=
nw_{n}|\Omega|\left(\frac{r^{n+4}}{n+4}+\tau\frac{r^{n+2}}{n+2}\right)-
(2\pi)^{n}\sum\limits_{l=1}^{m}\Gamma_{l}(\Omega)\int_{B_{r}}|\widehat{\phi}_{l}(z)|^{2}dz,
\end{eqnarray*}
which, together with (\ref{3-1}), implies
\begin{eqnarray} \label{3-5}
\Gamma_{m+1}(\Omega)\leq\inf_{r>2\pi\left(\frac{m}{w_{n}|\Omega|}\right)^{\frac{1}{n}}}\left\{
\frac{\frac{nw_{n}}{(2\pi)^{n}}|\Omega|\left(\frac{r^{n+4}}{n+4}+\tau\frac{r^{n+2}}{n+2}\right)-
\sum\limits_{l=1}^{m}\Gamma_{l}(\Omega)\int_{B_{r}}|\widehat{\phi}_{l}(z)|^{2}dz}
{\frac{w_{n}|\Omega|r^{n}}{(2\pi)^{n}}-\sum\limits_{j=1}^{m}\int_{B_{r}}|\widehat{\phi_{j}}(z)|^{2}dz}\right\}.
\end{eqnarray}
By Plancherel's Theorem, one has
\begin{eqnarray} \label{3-6}
\int_{B_{r}}|\widehat{\phi_{j}}(z)|^{2}\leq1
\end{eqnarray}
for each $j$. Applying \textbf{FACT} in Section \ref{s2}
(equivalently, the nonnegativity of eigenvalues), (\ref{3-6}) and
\cite[Lemma A1]{bcl} (see also \cite{lt}) to (\ref{3-5}) yields
\begin{eqnarray*}
\sum\limits_{j=1}^{m}\Gamma_{j}(\Omega)\leq\inf_{r>2\pi\left(\frac{m}{w_{n}|\Omega|}\right)^{\frac{1}{n}}}\left\{\frac{nw_{n}}{(2\pi)^{n}}|\Omega|\left(\frac{r^{n+4}}{n+4}
+\tau\frac{r^{n+2}}{n+2}\right)\right\}.
\end{eqnarray*}
The estimate in Theorem \ref{maintheorem} follows directly by
letting
$r\rightarrow2\pi\left(\frac{m}{w_{n}|\Omega|}\right)^{\frac{1}{n}}$.
\end{proof}

Finally, we have:

\begin{proof} [Proof of Corollary \ref{corollary}]
It follows from (\ref{3-5}) and (\ref{3-6}) that
\begin{eqnarray*}
\Gamma_{m+1}(\Omega)\leq\inf_{r>2\pi\left(\frac{m}{w_{n}|\Omega|}\right)^{\frac{1}{n}}}\left\{
\frac{\frac{nw_{n}}{(2\pi)^{n}}|\Omega|\left(\frac{r^{n+4}}{n+4}+\tau\frac{r^{n+2}}{n+2}\right)}
{\frac{w_{n}|\Omega|r^{n}}{(2\pi)^{n}}-m}\right\},
\end{eqnarray*}
which is the first estimate in Corollary \ref{corollary}. Define a
function $F(r)$ as
\begin{eqnarray*}
F(r):=\frac{\frac{nw_{n}}{(2\pi)^{n}}|\Omega|\left(\frac{r^{n+4}}{n+4}+\tau\frac{r^{n+2}}{n+2}\right)}
{\frac{w_{n}|\Omega|r^{n}}{(2\pi)^{n}}-m},\qquad
r>2\pi\left(\frac{m}{w_{n}|\Omega|}\right)^{\frac{1}{n}}.
\end{eqnarray*}
If $\tau=0$, then
\begin{eqnarray*}
F'(r)=\left[w_{n}|\Omega|r^{n}-m(2\pi)^{n}\right]^{-2}\cdot\Bigg{\{}nw_{n}|\Omega|r^{n+3}\left(w_{n}|\Omega|r^{n}-m(2\pi)^{n}\right)-
nw_{n}|\Omega|r^{n-1}\frac{n}{n+4}w_{n}|\Omega|r^{n+4}\Bigg{\}}.
\end{eqnarray*}
Solving the equation $F'(r)=0$ yields its solution
$r_{0}=2\pi\left(\frac{m(n+4)}{4w_{n}|\Omega|}\right)^{\frac{1}{n}}$,
The second estimate in Corollary \ref{corollary} follows directly by
using the fact $\Gamma_{m+1}(\Omega)\leq F(r_0)$.
\end{proof}

\section*{Acknowledgments}
\renewcommand{\thesection}{\arabic{section}}
\renewcommand{\theequation}{\thesection.\arabic{equation}}
\setcounter{equation}{0} \setcounter{maintheorem}{0}

This work is partially supported by the NSF of China (Grant Nos.
11801496 and 11926352), the Fok Ying-Tung Education Foundation
(China) and  Hubei Key Laboratory of Applied Mathematics (Hubei
University).


\begin{thebibliography}{9999}

\bibitem{bcl} B. Brandolini, F. Chiacchio,  J.-J. Langford, \emph{Estimates for sums of eigenvalues of the free plate via the fourier
transform}, Commun. Pure Appl. Anal. {\bf 19}(1) (2020) 113--122.

\bibitem{lmc1} L.-M. Chasman, \emph{An isoperimetric inequality for fundamental tones
of free plates}, Commun. Math. Phys. {\bf 303} (2011) 421--449.

\bibitem{lmc2} L.-M. Chasman, \emph{An isoperimetric inequality for fundamental tones of free plates with
nonzero Poisson's ratio}, Applicable Analysis {\bf 95} (2016)
1700--1735.


\bibitem{dmwxy} F. Du, J. Mao, Q.-L. Wang, C.-Y. Xia, Y. Zhao,
\emph{Estimates for eigenvalues of the Neumann and Steklov
problems}, avilable online at arXiv:1902.08998.

\bibitem{pk} P. Kr\"{o}ger, \emph{Estimates for sums of eigenvalues of the Laplacian}, J. Funct. Anal. {\bf 126}(1) (1994) 217--227.


\bibitem{lt} L. Li, L. Tang, \emph{Some upper bounds for sums of eigenvalues of the Neumann Laplacian}, Proc. Amer.
Math. Soc. {\bf134} (2006) 3301--3307.

\bibitem{gs1}  G. Szeg\H{o}, \emph{On membranes and plates}, Proc. Nat. Acad. Sci. {\bf 36} (1950) 210--216.

\bibitem{gs2}  G. Szeg\H{o}, \emph{Note to my paper ``On membranes and
plates"}, Proc. Nat. Acad. Sci. {\bf 44} (1958) 314--316.

\bibitem{hfw} H.-F. Weinberg, \emph{ An isoperimetric inequality for the $N$-dimensional free membrane
problem},  J. Rational Mech. Anal. {\bf 5}  (1956) 633--636.


\end{thebibliography}
\end{document}